\documentclass[a4paper,11pt]{article}

\usepackage{amssymb}
\usepackage{amsmath}
\usepackage{amsthm}
\usepackage[ansinew]{inputenc}
\usepackage{bbm}
\usepackage{enumerate}

\usepackage[numbers,square]{natbib}
\newtheorem{lemma}{Lemma}
\newtheorem{theorem}{Theorem}

\theoremstyle{definition}

\newtheorem{condition}{Condition}

\begin{document}

\title{Polar sets of anisotropic Gaussian random fields\thanks{This research was supported by the Deutsche Forschungsgemeinschaft through the SFB 649 "Economic Risk". The author expresses his appreciation of the guidance he has received from his thesis advisor, Markus Reiß.}}

\author{Jakob Söhl\thanks{E-mail address: soehl@math.hu-berlin.de}\\
Humboldt-Universität zu Berlin}

\date{November 13, 2009}

\maketitle

\begin{abstract}
This paper studies polar sets of anisotropic Gaussian random fields, i.e. sets which a Gaussian random field does not hit almost surely. The main assumptions are that the eigenvalues of the covariance matrix are bounded from below and that the canonical metric associated with the Gaussian random field is dominated by an anisotropic metric. We deduce an upper bound for the hitting probabilities and conclude that sets with small Hausdorff dimension are polar.
Moreover, the results allow for a translation of the Gaussian random field by a random field, that is independent of the Gaussian random field and whose sample functions are of bounded Hölder norm.
\end{abstract}


\section{Introduction}

Anisotropic Gaussian random fields arise naturally in stochastic partial differential equations, image processing, mathematical finance and other areas.
Let $X=\{X(t)|t\in I\subset \mathbbm{R}^N\}$ be a centered Gaussian random field with values in $\mathbbm{R}^d$, where $I$ is bounded. We will call $X$ an \emph{$(N,d)$-Gaussian random field}.
The distance in the \emph{canonical metric} associated with the Gaussian random field is $\sqrt{\mathbb{E}\left[\|X(s)-X(t)\|^2\right]}$, where $\|\cdot\|$ denotes the Euclidean metric.
Polar sets of Gaussian random fields are investigated in \cite{polar} under the assumptions that the components are independent copies of the same random field, that the variance is constant and that \(\sqrt{\mathbb{E}\left[\|X(s)-X(t)\|^2\right]}\le c \|s-t\|^\beta\) holds with constants $c,\beta>0$.
The recent works \cite{anisotropic} and \cite{hitting} consider the anisotropic metric
\begin{equation}
\rho(s,t):=\sum_{j=1}^N|s_j-t_j|^{H_j}
\end{equation}
with $H\in]0,1]^N$ and assume \(\sqrt{\mathbb{E}\left[\|X(s)-X(t)\|^2\right]}\le c \rho(s,t).\)
In addition they require for the variance only to be bounded from below.
In this paper the assumptions on the variance and on the independent copies in the components are substituted by the milder assumption that the eigenvalues of the covariance matrix are bounded from below.
The random fields in the components neither need to be identical distributed nor independent.
Hence, we require weaker assumptions on the dependency structure of the components of the Gaussian random field than \cite{polar}, \cite{anisotropic} and \cite{hitting}.
It follows from an upper bound on the hitting probabilities of $X$ that sets with Hausdorff dimension smaller than $d-\sum_{j=1}^{N}1/H_j$ are polar.
Our results allow for a translation of the Gaussian random field $X$ by a random field, that is independent of $X$ and whose sample functions are Lipschitz continuous with respect to the metric $\rho$.

As an application we show that an estimator in \cite{reiss}, which calibrates an exponential Lévy model by option data, is almost surely well-defined.

\section{Main results}

Let $X$ be an $(N,d)$-Gaussian random field. Recall that we suppose the index set $I$ to be bounded. We will assume the following two conditions.

\begin{condition}\label{metric}
There is a constant $c>0$ such that for all $s,t\in I$ we have $\sqrt{\mathbb{E}\left[\|X(s)-X(t)\|^2\right]}\le c \rho(s,t)$.
\end{condition}

\begin{condition}\label{eigenvalues}
There is a constant $\lambda>0$ such that for all $t\in I$ and for all $e\in \mathbbm{R}^d$ with $\|e\|=1$ we have
$\mathbb{E}[(\sum_{j=1}^d e_j X_j(t))^2]\ge \lambda$.
\end{condition}

Condition~\ref{metric} bounds the canonical metric in terms of the anisotropic metric~$\rho$. Condition~\ref{eigenvalues} bounds the eigenvalues of the covariance matrix from below. It excludes, for example, cases where $X$ takes values only in some vector subspace.

We will use a uniform modulus of continuity, see (69) in 
\cite[p.~167]{anisotropic}.
We restate this result in the next inequality. A weaker formulation suffices for our purpose and is proved in the Appendix.
Let $X$ be an $(N,d)$-Gaussian random field, that satisfies Condition~\ref{metric}.
Then there is a version $X'$ of $X$ and a constant $\tilde c>0$ such that almost surely the following inequality holds:
\begin{equation}\label{inequmodulus}
\limsup_{\epsilon \downarrow 0}\sup_{s,t\in I, \rho(s,t)\le \epsilon}\frac{\|X'(s)-X'(t)\|}{\epsilon \sqrt{\log(\epsilon^{-1})}}\le \tilde c.
\end{equation}

We will always assume that $X$ is a version, which satisfies~\eqref{inequmodulus}.
We define by $\text{Lip}_{\rho}(L):=\{f:I\rightarrow \mathbbm{R}^d|\: \|f(s)-f(t)\|\le L \rho(s,t) \: \forall s,t\in I\}$ the \emph{$L$-Lipschitz functions} with respect to the metric $\rho$.
In each direction $j$ the functions in $\text{Lip}_{\rho}(L)$ are Hölder continuous with exponent $H_j$. We denote by $B_{\rho}(t,r):=\{s\in\mathbbm{R}^N|\rho(s,t)\le r\}$ the closed ball of radius $r$ around $t$.

\begin{lemma}\label{ball}
Let $X$ be an $(N,d)$-Gaussian random field, that satisfies Conditions~\ref{metric} and \ref{eigenvalues}.
Then for each $L\ge 0$ there is a constant $C>0$ such that for all $t\in I$,
for all $r>0$
and for all functions $f\in \text{Lip}_{\rho}(L)$ we have
\begin{equation}\label{lemma}
\mathbb{P}\left(\inf_{s\in B_{\rho}(t,r)\cap I}\|X(s)-f(s)\|\le r\right)\le C r^d.
\end{equation}
\end{lemma}

\begin{proof}
For all integers $n\ge 1$ we define $\epsilon_n:=r \exp(-2^{n+1})$ and denote by $N_n:=N_{\rho}(B_{\rho}(t,r)\cap I,\epsilon_n)$ the covering number, that is the minimum number of $\rho$-balls with radii $\epsilon_n$ and centers in $B_{\rho}(t,r)\cap I$ that are needed to cover $B_{\rho}(t,r)\cap I$. We have the inclusion $B_{\rho}(t,r)\subseteq \prod_{j=1}^N[t_j-r^{1/H_j},t_j+r^{1/H_j}]$.
On the other hand each set $\prod_{j=1}^N[s_j,s_j+(\epsilon_n/N)^{1/H_j}]$ can be covered by a single ball with radius $\epsilon_n$.
Hence there is a constant $c_1>0$ independent of $n$ such that $N_n\le \prod_{j=1}^N ((2rN/\epsilon_n)^{(1/H_j)}+1)\le c_1\exp(Q2^{n+1})$ where $Q=\sum_{j=1}^N 1/H_j$.

We denote by $\{t_i^{(n)}\in B_{\rho}(t,r)\cap I|1\le i\le N_n\}$ a set of points such that the balls with the centers $\{t_i^{(n)}\}$ and radii $\epsilon_n$ cover $B_\rho(t,r)\cap I$. We define \[r_n:=\beta \epsilon_n 2^{\frac{n+1}{2}},\]
 where $\beta> \tilde c$ is some constant to be determined later. For all integers $n,k\ge 1$ and $1\le i \le N_k$, we define the following events
\begin{align}
A_i^{(k)}&:=\Bigl\{\|X(t_i^{(k)})-f(t_i^{(k)})\|\le  r+\sum_{l=k}^\infty r_l\Bigr\},\\
A^{(n)}&:=\bigcup_{k=1}^n\bigcup_{i=1}^{N_k}A_i^{(k)}=A^{(n-1)}\cup \bigcup_{i=1}^{N_n}A^{(n)}_i,\label{A}
\end{align}
where the last equality only holds for $n\ge2$.
We will show that the probability in \eqref{lemma} can be dominated by the limit of the probabilities of the sets $A^{(n)}$
\begin{equation}\label{keystep}
\mathbb{P}\left(\inf_{s\in B_{\rho}(t,r)\cap I}\|X(s)-f(s)\|\le r\right)\le \lim_{n \rightarrow \infty }\mathbb{P}(A^{(n)}).
\end{equation}
For all $s\in B_{\rho}(t,r)\cap I$ and all $n\ge 1$ there exists $i_n$ such that $\rho(s,t_{i_n}^{(n)})\le \epsilon_n$. By \eqref{inequmodulus} we obtain almost surely
\[\limsup_{n\rightarrow \infty} \, \sup_{s\in I}\frac{\|X(s)-X(t_{i_n}^{(n)})\|}{r_n}\le \frac{\tilde c}{\beta}<1,\]
where the supremum over $s$ is to be understood such that $i_n$ varies according to $s$. Let $\kappa\in\; ]\tilde c/\beta,1[.$
Especially there is $N$ such that for all $n\ge N$ we have
\begin{equation}\label{supremum}
\sup_{s\in I}\frac{\|X(s)-X(t_{i_n}^{(n)})\|}{r_n}\le \kappa.
\end{equation}
By going over to a possibly greater constant $N$ we ensure that $(1-\kappa)\tilde c 2^{\frac{N+1}{2}}\ge L$.
On the event $\inf_{s\in B_{\rho}(t,r)\cap I}\|X(s)-f(s)\|\le r$ there exists $s_0\in B_{\rho}(t,r)\cap I$ such that
\begin{equation}\label{inf}
\|X(s_0)-f(s_0)\|\le r+\sum_{l=N+1}^{\infty}r_l.
\end{equation}
Choose $i_N$ such that $\rho(s_0,t_{i_N}^{(N)}) \le \epsilon_N$.
Using \eqref{supremum}, \eqref{inf} and the Lipschitz continuity of $f$ we obtain
\begin{align*}
&\|X(t_{i_N}^{(N)})-f(t_{i_N}^{(N)})\|\displaybreak[0]\\
\le&\|X(t_{i_N}^{(N)})-X(s_0)\|+\|X(s_0)-f(s_0)\| +\|f(s_0)-f(t_{i_N}^{(N)})\|\displaybreak[0]\\
\le& \kappa r_N+r+\sum_{l=N+1}^{\infty}r_l+L\rho(s_0,t_{i_N}^{(N)})\displaybreak[0]\\
\le& \kappa r_N+r+\sum_{l=N+1}^{\infty}r_l+(1-\kappa)\tilde c 2^{\frac{N+1}{2}} \epsilon_N
\le r+\sum_{l=N}^{\infty}r_l
\end{align*}
and \eqref{keystep} is established.

Trivially we have for $n\ge 2$
\[\mathbb{P}(A^{(n)})\le \mathbb{P}(A^{(n-1)}) + \mathbb{P}(A^{(n)}\backslash A^{(n-1)})\]
and by \eqref{A} we have
\[\mathbb{P}(A^{(n)}\backslash A^{(n-1)})\le \sum_{i=1}^{N_n}\mathbb{P}(A_i^{(n)}\backslash A_{i'}^{(n-1)}),\]
where $i'$ is chosen such that $\rho(t_i^{(n)},t_{i'}^{(n-1)})<\epsilon_{n-1}$.
We note that for $n\ge 2$
\begin{align}
 &\mathbb{P}(A_i^{(n)}\backslash A_{i'}^{(n-1)})\label{probability}\\
=&\mathbb{P}\left( \|X(t_i^{(n)})-f(t_i^{(n)})\|  \le r+\sum_{l=n}^{\infty}r_l,  \|X(t_{i'}^{(n-1)})-f(t_{i'}^{(n-1)})\| > r+\sum_{l=n-1}^{\infty}r_l\right)\notag\\
\le&\mathbb{P}\left( \|X(t_i^{(n)})-f(t_i^{(n)})\|  \le c_2\,r,  \|X(t_{i}^{(n)})-X(t_{i'}^{(n-1)})\| > r_{n-1}-L\epsilon_{n-1}\right)\notag\\
\le&\mathbb{P}\left( \|X(t_i^{(n)})-f(t_i^{(n)})\|  \le c_2\,r,  \|X(t_{i}^{(n)})-X(t_{i'}^{(n-1)})\| > (\beta 2^{\frac{n}{2}}-L)\epsilon_{n-1}\right),\notag
\end{align}
where $c_2=1+\beta \sum_{l=1}^{\infty}2^{\frac{l+1}{2}}\exp(-2^{l+1})$.
We ensure $(\beta 2^{\frac{n}{2}}-L)>0$ by choosing $\beta>L/2$.
The idea is to rewrite $X(t_{i}^{(n)})-X(t_{i'}^{(n-1)})$ as a sum of two terms, one expressed by $X(t_{i}^{(n)})$ and the other independent of $X(t_{i}^{(n)})$.

$\text{Lip}_{\rho}(L)$ is invariant under orthogonal transformations.
By the spectral theorem we may choose new coordinates such that the covariance matrix at $t_i^{(n)}$ is diagonal.
Then the components of $X(t_i^{(n)})$ are independent.
By assumption $\sigma_j(s):=\sqrt{\mathbb{E}\left[X_j(s)^2\right]}>0$.
We define the standard normal random variables
\[Y_j(s):=\frac{X_j(s)}{\sigma_j(s)}.\]
Note that $\mathbb{E}[Y(t_i^{(n)})Y(t_i^{(n)})^{\top}]=\mbox{Id}$ holds.
If $\mathbb{E}\left[(X_j(s)-X_j(t))^2\right]>0$ we define \[Y_j(s,t):=\frac{X_j(s)-X_j(t)}{\sqrt{\mathbb{E}\left[(X_j(s)-X_j(t))^2\right]}}\]
and $Y_j(s,t):=0$ otherwise. We further define a matrix $\eta$ and a random vector $Z$ by
\begin{align*}
\eta&:=\mathbb{E}\left[Y(t_i^{(n)},t_{i'}^{(n-1)})Y(t_i^{(n)})^{\top}\right],\\
Z(t_i^{(n)},t_{i'}^{(n-1)})&:=Y(t_i^{(n)},t_{i'}^{(n-1)})-\eta Y(t_i^{(n)}).
\end{align*}
We observe that $|\eta_{jk}|\le 1$ and hence in the operator norm $\|\eta\|\le d$.
The random vectors $Z(t_i^{(n)},t_{i'}^{(n-1)})$ and $Y(t_i^{(n)})$ are independent because the covariance matrix is the zero matrix.
By the definition of $Y(t_i^{(n)})$ we see that $Z(t_i^{(n)},t_{i'}^{(n-1)})$ and $X(t_i^{(n)})$ are independent, too.
We want to bound $\mathbb{P}(A_i^{(n)}\backslash A_{i'}^{(n-1)})$. If $t_i^{(n)}=t_{i'}^{(n-1)}$ then $\mathbb{P}(A_i^{(n)}\backslash A_{i'}^{(n-1)})=0$ holds. Thus we may assume that $\rho(t_i^{(n)},t_{i'}^{(n-1)})>0$.
\eqref{probability}~is bounded by
\begin{align*}
& \mathbb{P}\left( \|X(t_i^{(n)})-f(t_i^{(n)})\|  \le c_2\,r, \; \|Y(t_{i}^{(n)},t_{i'}^{(n-1)})\| > \frac{(\beta 2^{\frac{n}{2}}-L)\epsilon_{n-1}}{c \, \rho(t_{i}^{(n)},t_{i'}^{(n-1)})}\right)\\
\le & \mathbb{P}\left( \|X(t_i^{(n)})-f(t_i^{(n)})\|  \le c_2\,r, \; \|Z(t_i^{(n)},t_{i'}^{(n-1)})\|+\|\eta Y(t_i^{(n)})\| >
\frac{\beta 2^{n/2}-L}{c} \right)\displaybreak[0]\\
\le& \mathbb{P}\left( \|X(t_i^{(n)})-f(t_i^{(n)})\|  \le c_2\,r, \; \|Z(t_i^{(n)},t_{i'}^{(n-1)})\| >
\frac{\beta 2^{n/2}-L}{2c} \right)\\
+&\mathbb{P}\left( \|X(t_i^{(n)})-f(t_i^{(n)})\|  \le c_2\,r, \; d\| Y(t_i^{(n)})\| >
\frac{\beta 2^{n/2}-L}{2c} \right)\\
:=&I_1+I_2.
\end{align*}
Each component of $Z(t_i^{(n)},t_{i'}^{(n-1)})$ is a weighted sum of at most $d+1$ standard normal random variables with weights in $[-1,1]$. Hence the variance of each component is at most $(d+1)^2$. In the following $c_l$  with $l \in \mathbbm{N}$ will denote positive constants.
By the independence of $X(t_i^{(n)})$ and $Z(t_i^{(n)},t_{i'}^{(n-1)})$ we have
\begin{align*}
I_1=&\mathbb{P}\left( \|X(t_i^{(n)})-f(t_i^{(n)})\|  \le c_2 \, r\right) \mathbb{P} \left( \|Z(t_i^{(n)},t_{i'}^{(n-1)})\| > \frac{\beta 2^{n/2}-L}{2c} \right)\displaybreak[0]\\
\le& c_3 \, r^d \mathbb{P} \left( \|Z(t_i^{(n)},t_{i'}^{(n-1)})\| > \frac{\beta 2^{n/2}-L}{2c} \right)\displaybreak[0]\\
\le& c_3 \, r^d  \,  \frac{2d}{\sqrt{2\pi}}\frac{2\sqrt{d}(d+1)c}{\beta 2^{\frac{n}{2}}-L} \exp\left(-\frac{(\beta 2^{\frac{n}{2}}-L)^2}{8 c^2 d (d+1)^2}\right)\displaybreak[0]\\
\le& c_4 \, r^d \exp\left(-\frac{(\beta 2^{\frac{n}{2}}-L)^2}{8c^2d (d+1)^2}\right).
\end{align*}
By the definition of $Y(t_i^{(n)})$ we have with the abbreviation
$\sigma_j=\sigma_j(t_i^{(n)})$
\begin{align*}
I_2\le&\int_{\{\|u-f(t_i^{(n)})\|\le c_2 \, r, \; \|(\frac{u_k}{\sigma_k})_k\|>
\frac{\beta 2^{n/2}-L}{2dc}\}}
\left(\frac{1}{2\pi}\right)^\frac{d}{2}\frac{1}{\sigma_1\cdots\sigma_d} e^{-\frac{1}{2}\left(\frac{u_1^2}{\sigma_1^2}+\dots+\frac{u_d^2}{\sigma_d^2}\right)}\mathrm{d}u\displaybreak[0]\\
\le&\int_{\{\|u-f(t_i^{(n)})\|\le c_2 \, r\}}
\left(\frac{1}{2\pi}\right)^\frac{d}{2}\frac{1}{\sigma_1\cdots\sigma_d} e^{-\frac{1}{4}\left(\frac{u_1^2}{\sigma_1^2}+\dots+\frac{u_d^2}{\sigma_d^2}\right)}\mathrm{d}u\;
e^{-\frac{1}{4}\left(\frac{(\beta 2^{\frac{n}{2}}-L)^2}{4d^2c^2}\right)}\displaybreak[0]\\
\le&c_5 \, r^d \;
\exp\left(-\frac{(\beta 2^{\frac{n}{2}}-L)^2}{16d^2c^2}\right).\\
\intertext{To sum it up}
\mathbb{P}(A^{(n)})\le&\mathbb{P}(A^{(n-1)})+c_6\, r^d N_n \exp \left(-\frac{(\beta 2^{\frac{n}{2}}-L)^2}{16d(d+1)^2c^2} \right)\\
\le&\mathbb{P}(A^{(1)})+c_6\, r^d \sum_{k=2}^\infty N_k \exp \left(-\frac{(\beta 2^{\frac{k}{2}}-L)^2}{16d(d+1)^2c^2} \right)\\
\le&c_7r^d+c_6\, r^d \sum_{k=2}^\infty \,c_1 \, \exp \left(Q2^{k+1}-\frac{(\beta 2^{\frac{k}{2}}-L)^2}{16d(d+1)^2c^2} \right).
\end{align*}
We choose $\beta>\max(\tilde c,L/2)$ such that $\frac{\beta^2}{16d(d+1)^2c^2} >2Q$. Then the sum converges and the lemma follows by \eqref{keystep}.
\end{proof}

Let $\mathcal{H}_\alpha$ denote the $\alpha$-dimensional \emph{Hausdorff measure}, see the definition, for instance, in 
\cite[p. 129]{kahane}.
Recall that $Q=\sum_{j=1}^N 1/H_j$ with $H_j$ as in the definition of the metric $\rho$.

\begin{theorem}\label{polar}
Let $X$ be an $(N,d)$-Gaussian random field that satisfies Conditions~\ref{metric} and \ref{eigenvalues}.
If $Q<d$, then for each $L\ge 0$ there is a constant $C>0$ such that all Borel sets $F\subseteq \mathbb{R}^d$ and all random fields $Y$ which are independent of $X$ and whose sample functions are all in $\text{Lip}_{\rho}(L)$ satisfy
\begin{equation}
\mathbb{P}\left(\exists s\in I : X(s)+Y(s) \in F  \right)\le C \mathcal{H}_{d-Q}(F).
\end{equation}
\end{theorem}

\begin{proof}
By Fubini's theorem it suffices to show for all functions $f\in \text{Lip}_{\rho}(L)$
\[\mathbb{P}\left(\exists s\in I : X(s)+f(s) \in F  \right)\le C \mathcal{H}_{d-Q}(F).\]
We choose some constant $\gamma>\mathcal{H}_{d-Q}(F)$. By definition of the Hausdorff measure there is a set of balls $\{B(x_l,r_l):l=0,1,2,\dots\}$ such that
\begin{equation}\label{hausdorffmeasure}
F \subseteq \bigcup_{l=0}^\infty B(x_l,r_l) \quad \mbox{and} \quad  \sum_{l=0}^\infty(2r_l)^{d-Q}\le \gamma.
\end{equation}
For all $j$ we cut the bounded index set $I$ orthogonal to the $j$-axis with distance $(r_l/N)^{1/H_j}$ between the cuts. Each piece of $I$ can be covered by a single ball of radius $r_l$ in the metric $\rho$. Hence there is a constant $c_8>0$ such that $I$ can be covered by at most $c_8r_l^{-Q}$ balls.
We apply Lemma~\ref{ball} to these balls. By summing up we obtain
\begin{equation}\label{wholeindexset}
\mathbb{P}\left(\exists s\in I : X(s)+f(s) \in B(x_l,r_l)   \right)\le c_9 r_l^{d-Q}.
\end{equation}
By \eqref{hausdorffmeasure} and \eqref{wholeindexset} we have
\begin{align*}
&\mathbb{P}\left(\exists s\in I : X(s)+f(s) \in F   \right)\\
\le& \sum_{l=0}^\infty \mathbb{P}\left(\exists s\in I : X(s)+f(s) \in B(x_l,r_l)  \right)\le c_{10}\gamma.
\end{align*}
We have $\mathbb{P}\left(\exists s\in I : X(s)+f(s)\in F \right)\le c_{10}\mathcal{H}_{d-Q}(F)$, since $\gamma>\mathcal{H}_{d-Q}(F)$ was chosen arbitrarily.
\end{proof}

\section{Application}

In this section we apply our results to nonparametric statistics. Observations in the Gaussian white noise model lead to a Gaussian process to which then Theorem~\ref{polar} is applied. We conclude that a certain estimator is almost surely well-defined.
Existing results on Gaussian random fields have too restrictive assumptions for the application considered here.

The common approach to the estimation of Lévy processes is to take advantage of the Lévy-Khintchine representation. This often involves taking a complex logarithm of the empirical characteristic function, see e.g. \cite{reiss} and \cite{gugushvili}.
The branch of the logarithm needs to be taken such that the $\log$-characteristic function is continuous. This is referred to as the \emph{distinguished logarithm}.
To this end the empirical characteristic function has to be distinct from zero everywhere which means that zero needs to be a polar set.
For instance in the set-up of \cite{reiss} the following function is estimated:
\[\psi:\mathbbm{R}\rightarrow \mathbbm{C},v\mapsto\frac{1}{T}\log\left(1+iv(1+iv)\mathcal{FO}(v)\right),\]
where $T>0$, $\mathcal{O}$ is some $L^1(\mathbbm{R})$ function modeling option prices and $\mathcal{F}$ denotes the Fourier transform. The argument of the logarithm is distinct from zero for all $v\in\mathbbm{R}$ and the logarithm is taken such that $\psi$ is continuous with $\psi(0)=0$. The function $\mathcal{O}(x)$ is dominated by an exponentially decaying function for $|x|\rightarrow \infty$.
Especially $x\mathcal{O}(x)$ is integrable.
To simplify matters we suppose that $\mathcal{O}$ is continuously observed as in the Gaussian white noise model, cf. \cite{reiss}. For $\epsilon\in L^2(\mathbbm{R})$ we observe
\[\mathrm{d}\tilde{\mathcal{O}}(x)=\mathcal{O}(x)\mathrm{d}x+ \epsilon(x) \, \mathrm{d}W(x) \qquad \mbox{for all } x\in\mathbbm{R}. \]
Taking the Fourier transform yields
\begin{align*}
\mathcal{F}(\mathrm{d}\tilde{\mathcal{O}})(v)&=\mathcal{FO}(v)+\int_{-\infty}^\infty e^{ivx} \epsilon(x) \, \mathrm{d}W(x).
\end{align*}
Thus the canonical estimator for $\psi$ is
\[\tilde \psi (v):=\frac{1}{T}\log\left( 1+iv(1+iv)\mathcal{F} (\mathrm{d}\tilde{\mathcal{O}})(v) \right) \qquad \mbox{for all } v\in \mathbbm{R}.\]
In \cite{reiss} a trimmed log-function is used to bound the real part. But this is of no importance for the question whether the estimator $\tilde \psi$ is well-defined.
We require the following condition on $\epsilon$.

\begin{condition}\label{tail}
There is a $p>1$ such that $\int_{-\infty}^{\infty}(1+|x|)^{p}\epsilon(x)^2\mathrm{d}x<\infty$.
\end{condition}

For example, if $\epsilon\in L^2(\mathbbm{R})$ and $\epsilon(x)=O(|x|^{-p})$ for $|x|\rightarrow \infty$ with $p>1$, then the condition is satisfied. Condition~\ref{tail} and Lemma~\ref{continuity} in the Appendix imply the uniform modulus of continuity \eqref{inequmodulus} for a version of $X(v):=\int_{-\infty}^\infty e^{ivx} \epsilon(x) \, \mathrm{d}W(x)$. We will assume that $X$ is a version that satisfies~\eqref{inequmodulus}. Thus in the definition of $\tilde \psi$ the argument of the logarithm is almost surely continuous.

\begin{lemma}
Let $\epsilon$ fulfill Condition~\ref{tail}. Let $\mathcal{O}\in L^1(\mathbbm{R})$ such that $x\mathcal{O}(x)$ is integrable. Then
the estimator $\tilde \psi$ is almost surely well-defined.
\end{lemma}

\begin{proof}
We have to show that almost surely the argument of the logarithm does not hit zero. The process $1+iv(1+iv)\mathcal{F} (\mathrm{d}\tilde{\mathcal{O}})(v)$ equals $1$ at $v=0$. It suffices to consider the process on $\mathbbm{R}\backslash \{0\}$. We rewrite the process as \[iv(1+iv)\left(\frac{1}{iv(1+iv)}+\mathcal{FO}(v)+\int_{-\infty}^\infty e^{ivx} \epsilon(x) \, \mathrm{d}W(x)\right).\]
We define
\[f(v):=\frac{1}{iv(1+iv)}+\mathcal{FO}(v)\quad\text{and}\quad X(v):=\int_{-\infty}^\infty e^{ivx} \epsilon(x) \, \mathrm{d}W(x).\]
On a bounded index set X is an (1,2)-Gaussian random field. We will apply Theorem~\ref{polar} to $X$, $Y=f$ and $F=\{0\}$.
It is proved in the Appendix that under Condition~\ref{tail}
there is a constant $c>0$ such that for all $u,v\in \mathbbm{R}$ the inequality
\begin{equation}\label{rho}
\sqrt{\mathbb{E}[\|X(u)-X(v)\|^2]}\le c |u-v|^{\min(p/2,1)}.
\end{equation}
holds. This gives reason to the definition $\rho(u,v):=|u-v|^H$ with $H=\min(p/2,1)\in\;]1/2,1]$. Thus Condition~\ref{metric} is satisfied and we have $d-Q=2-1/H>0$.

It remains to show that Condition~\ref{eigenvalues} is fulfilled and that $f$ is Lipschitz continuous with respect to the metric $\rho$.
For $\epsilon=0\in L^2(\mathbbm{R})$ we have $\tilde \psi=\psi$ and thus $\tilde \psi$ is well-defined.
We will now show that the covariance matrix of $X(v)$ is not degenerated if $\epsilon\ne 0 \in L^2(\mathbb{R})$ and $v\ne 0$. Let $e\in\mathbbm{R}^2$ such that $e_1^2+e_2^2=1$. Then there is $\varphi\in[0,2\pi]$ such that $e_1=\sin \varphi$ and $e_2=\cos \varphi$. Consider $X$ as a $\mathbbm{R}^2$-valued stochastic process. The Itô isometry yields
\begin{align*}
\mathbb{E}[(e_1X_1(v)+e_2X_2(v))^2] &=\mathbb{E}\left[\left(\int_{-\infty}^\infty (e_1\cos(vx)+e_2\sin(vx)) \epsilon(x) \, \mathrm{d}W(x)\right)^2\right]\\
&=\int_{-\infty}^\infty (e_1\cos(vx)+e_2\sin(vx))^2 \epsilon(x)^2 \, \mathrm{d}x\\
&=\int_{-\infty}^\infty (\sin(\varphi+vx))^2 \epsilon(x)^2 \, \mathrm{d}x>0.
\end{align*}
The function
\[\mathbbm{R}\times[0,2\pi] \rightarrow\mathbbm{R}, (v,\varphi)\mapsto\int_{-\infty}^{\infty}(\sin(\varphi+vx))^2\epsilon(x)^2\mathrm{d}x\]
is continuous by dominated convergence.
On $([-V,-1/V]\cup[1/V,V])\times[0,2\pi]$ it takes a minimum $\lambda_V>0$ for $V>0$. Hence Condition~\ref{eigenvalues} is fulfilled on the index set $I_V=[-V,-1/V]\cup[1/V,V]$.

Since $x\mathcal{O}(x)$ is integrable we have that $\mathcal{FO}$ is Lipschitz continuous on~$\mathbbm{R}$. $1/(iv(1+iv))$ is Lipschitz continuous on sets bounded away from zero. Hence $f$ is Lipschitz continuous on~$I_V$. Since $I_V$ is bounded it follows that $f$ is Lipschitz continuous with respect to the metric $\rho$ on $I_V$.

Thus we may apply Theorem~\ref{polar} to the index set $I_V=[-V,-1/V]\cup[1/V,V]$. Since $\mathcal{H}_{d-Q}(\{0\})=0$ we obtain
\(\mathbb{P}\left(\exists v\in I_V : X(v)+f(v) = 0   \right)=0.\)
Because $V>0$ was chosen arbitrarily the Lemma follows.
\end{proof}

\section*{Appendix}

\begin{lemma}
Let $X$ be an $(N,d)$-Gaussian random field that fulfills Condition~\ref{metric}.
Then there is a version $X'$ of $X$ and a constant $\tilde c>0$ with the following property. If $\epsilon_n\downarrow 0$ satisfies $\sum_{n=1}^{\infty}\epsilon_n^\vartheta <\infty$ for each $\vartheta>0$ then
\[\limsup_{n \rightarrow \infty}\sup_{s,t\in I, \rho(s,t)\le \epsilon_n}\frac{\|X'(s)-X'(t)\|}{\epsilon_n \sqrt{\log(\epsilon_n^{-1})}}\le \tilde c\]
holds almost surely.
\end{lemma}

\begin{proof}
We define the Gaussian random field $Y=\{Y(t,s)|t,(t+s)\in I\}$ by $Y(t,s):=X(t+s)-X(t)$. By Condition~\ref{metric} we have
\begin{align*}
d((t,s),(t',s'))&:=\sqrt{\mathbb{E}[\|Y(t,s)-Y(t',s')\|^2]}\\
&\le c_1 \min(\rho(0,s)+\rho(0,s'),\; \rho(s,s')+\rho(t,t')).
\end{align*}
Let $\epsilon \in \; ]0,1/4[$. We denote by $D_\epsilon$ the diameter of $T_\epsilon:=\{(t,s)\in I\times B_\rho(0,\epsilon)|t+s\in I\}$, where $B_{\rho}(0,\epsilon):=\{s\in\mathbbm{R}^N|\rho(s,0)\le \epsilon\}$. By the first argument of the minimum $D_\epsilon \le 2c_1\epsilon$ holds.
For $\delta\in\;]0,\epsilon[$ denote by $N_d(T_\epsilon,\delta)$ the covering number of $T_\epsilon$ in the metric $d$, i.e. the minimum number of d-balls with radius $\delta$ that are needed to cover $T_\epsilon$.
By the second argument of the minimum the covering number satisfies
$N_d(T_\epsilon,\delta)\le   c_2 \epsilon^{Q}/\delta^{2Q}$. By the previous estimates we obtain with the substitution $y=\delta/(c_2^{1/2Q}\sqrt{\epsilon})$
\[\int_0^{D_\epsilon}\sqrt{\log N_d(T_\epsilon,\delta)}\mathrm{d}\delta
\le \sqrt{2Q}c_2^{1/2Q}\sqrt{\epsilon} \int_0^{2c_1\sqrt{\epsilon}/c_2^{1/2Q}}\sqrt{\log{y^{-1}}} \,\mathrm{d}y.\]
By choosing a possibly larger constant $c_1$ or $c_2$ we may assume $2c_1=c_2^{1/2Q}$.
The integral is solved by
\[\int_0^x \sqrt{\log y^{-1}}\,\mathrm{d}y=\frac{\sqrt{\pi}}{2}-\frac{\sqrt{\pi}}{2}\text{Erf}(\sqrt{\log x^{-1}})+x\sqrt{\log x^{-1}},\]
where $\text{Erf(y)}=\frac{2}{\sqrt{\pi}}\int_0^y e^{-t^2}\mathrm{d}t$.
Estimating $\text{Erf}(y)$ from below yields some constant $c_3>0$ such that for all $x\in\;]0,1/2[$
\[\int_0^x \sqrt{\log y^{-1}}\,\mathrm{d}y
\le c_3 x\sqrt{\log x^{-1}}.\]
Thus for $\epsilon\in\;]0,1/4[$
\[\int_0^{D_\epsilon}\sqrt{\log N_d(T_\epsilon,\delta)}\mathrm{d}\delta\le c_4 \epsilon\sqrt{\log \epsilon^{-1/2}}\]
holds. By Lemma 2.2 in \cite{packing} there is a constant $K>0$ such that for all $u\ge2K c_4 \epsilon\sqrt{\log \epsilon^{-1/2}}$ we obtain
\[\mathbb{P}\left(\sup_{(t,s)\in T_\epsilon}\|X'(t+s)-X'(t)\|\ge u\right)\le\exp\left(-\frac{u^2}{(2KD_\epsilon)^2}\right).\]
Using the Borel-Cantelli lemma the statement follows.
\end{proof}

\begin{lemma}\label{continuity}
Let $\epsilon$ fulfill Condition~\ref{tail}.
Then there exists a number $c>0$ such that for all $u,v\in \mathbbm{R}$ the stochastic process $X(v)=\int_{-\infty}^\infty e^{ivx} \epsilon(x) \, \mathrm{d}W(x)$ satisfies
$\sqrt{\mathbb{E}[\|X(u)-X(v)\|^2]}\le c |u-v|^{\min(p/2,1)}$.
\end{lemma}
\begin{proof}
Condition~\ref{tail} is satisfied for $q:=\min(p,2)$ as well. We conclude that:
\begin{align*}
&\mathbb{E}\left[\|X(u)-X(v)\|^{2}\right]\\
=&\mathbb{E}\left[\left| \int_{-\infty}^{\infty}(e^{iux}-e^{ivx})\epsilon(x)\mathrm{d}W(x) \right|^{2}\right]\displaybreak[0]\\
=&\int_{-\infty}^{\infty} |e^{iux}-e^{ivx}|^2\epsilon(x)^2\mathrm{d}x\displaybreak[0]\\
\le&\int_{-\infty}^{\infty}\min(4,(u-v)^2x^2)\epsilon(x)^2\mathrm{d}x\displaybreak[0]\\
=&\int_{|x| \ge 2|u-v|^{-1}}4  \epsilon(x)^2\mathrm{d}x
   + \int_{|x| < 2|u-v|^{-1}}(u-v)^2 x^2 \epsilon(x)^2\mathrm{d}x\displaybreak[0]\\
\le&\int_{|x| \ge 2|u-v|^{-1}}4 \left(\frac{|x|}{2|u-v|^{-1}}\right)^{q} \epsilon(x)^2\mathrm{d}x\\
  & + \int_{|x| < 2|u-v|^{-1}}\left(\frac{2|u-v|^{-1}}{|x|}\right)^{2-q}(u-v)^2 x^2 \epsilon(x)^2\mathrm{d}x\\
=&2^{2-q}|u-v|^{q}\int_{-\infty}^{\infty}|x|^{q}\epsilon(x)^2\mathrm{d}x
\end{align*}
This shows the lemma.
\end{proof} 

\bibliography{gauss} 

\end{document}